\documentclass[preprint,9pt]{elsarticle}

\usepackage{amssymb}
\usepackage{amsmath}
\usepackage{amsthm}

\newtheorem{lemma}{Lemma}[section]
\newtheorem{thm}[lemma]{Theorem}
\newtheorem{prop}[lemma]{Proposition}
\newtheorem{cor}[lemma]{Corollary}
\newtheorem{rem}[lemma]{Remark}
\newtheorem{hyp}[lemma]{Hypothesis}

\journal{Journal of Group Theory}
\begin{document}

\setlength{\parindent}{0mm}

\newcommand{\m}{$\,\textrm{max}\,$}
\newcommand{\w}{\widehat}
\newcommand{\wi}{\widehat}
\newcommand{\ov}{\overline}
\newcommand{\N}{\mathbb{N}}
\def \P{\mathbb{P}}

\newcommand{\E}{\mathcal{E}}
\newcommand{\K}{\mathcal{K}}
\newcommand{\sym}{\mathcal{S}}
\newcommand{\A}{\mathcal{A}}

\newcommand{\Z}{\mathbb{Z}}

\newcommand{\wt}{\widetilde}
\newcommand{\wh}{\widehat}
\newcommand{\ti}{\tilde}

\newcommand{\M}{\textrm{M}}
\newcommand{\J}{\textrm{J}}
\newcommand{\ch}{\textrm{char}}
\newcommand{\sy}{\,\textrm{Syl}}
\newcommand{\au}{\textrm{Aut}}
\newcommand{\PSL}{\textrm{PSL}}
\newcommand{\PSU}{\textrm{PSU}}
\newcommand{\PGL}{\textrm{PGL}}
\newcommand{\PGaL}{P\Gamma L}
\newcommand{\GL}{\textrm{GL}}
\newcommand{\GU}{\textrm{GU}}
\newcommand{\Sp}{\textrm{Sp}}
\newcommand{\PSp}{\textrm{PSp}}
\newcommand{\Sz}{\textrm{Sz}}
\newcommand{\SL}{\textrm{SL}}
\newcommand{\SU}{\textrm{SU}}
\newcommand{\F}{$\textrm{GF}$}
\newcommand{\C}{\textrm{C}}
\newcommand{\FO}{\textrm{fix}_{\Omega}}
\newcommand{\FL}{\textrm{fix}_{\Lambda}}
\newcommand{\FD}{\textrm{fix}_{\Delta}}
\newcommand{\fixO}{\textrm{fix}_{\Omega}}
\newcommand{\fixL}{\textrm{fix}_{\Lambda}}
\newcommand{\out}{$\textrm{Out}$}
\newcommand{\Sym}{\mathcal{S}}
\newcommand{\Alt}{\mathcal{A}}
\newcommand{\rank}{$\textrm{r}$}
\newcommand{\He}{$\textrm{He}$}
\newcommand{\aut}{\textrm{Aut}}

\def \<{\langle }
\def \>{\rangle }
\def \L{\mathcal{L}}

\newcommand{\PGU}{\textrm{PGU}}
\newcommand{\Inndiag}{\textrm{Inndiag}}

\begin{frontmatter}

\title{Corrigendum to ``Transitive permutation groups with trivial four point stabilizers''}

\author{Paula H\"ahndel, Christoph M\"oller and Rebecca Waldecker}

\begin{abstract}
This article revisits earlier work by the last author together with Kay Magaard. We correct three mistakes in total -- one in the 
results about the Sylow 2- and 3-structure of groups that act with fixity 3, respectively, and another one in a result about possible components. We state and prove corrected versions of these results, we discuss potential consequences of the mistakes and how we address them, and we improve some results and add details.

\end{abstract}

\end{frontmatter}


\section{Introduction}

\vspace{0.2cm}
The motivation to investigate permutation groups that act with low fixity
stems from applications to Riemann surfaces.
Following Ronse (see \cite{Ro1980}), we say that \textbf{a group $G$ has fixity $k \in \N$ on a set $\Omega$} if and only if
$k$ is the maximum number of fixed points of elements of $G^\#$ on $\Omega$.
More background on our motivation and on applications can be found in \cite{MW2}, and in this article we correct mistakes in \cite{MW3}. There are three mistakes in total, which we address individually, and we also discuss where these faulty results were used in the original paper and whether or not this has further consequences.

The first error occurred when we analysed the $2$-structure, in Lemma 2.17 of \cite{MW3}, with the second, similar mistake in Lemma 2.20 of \cite{MW3}. Fortunately, this did not affect the classification of finite simple groups that act with fixity 3 and hence subsequent work. There was a third mistake when it comes to a more general situation and possible components, in Lemma 2.21 of \cite{MW3}.

We will state corrections of these results below. Later in this article, when we prove corrected versions of the results, then we will also point to all places in the original paper when the faulty results were used and discuss the consequences.


Here we have the main hypothesis for the remainder of this paper and then our main corrections:

\begin{hyp}\label{hypkfix}
Suppose that $G$ is a finite, transitive permutation
group with permutation domain $\Omega$. Suppose further that $G$ acts with fixity~$3$ on $\Omega$.
\end{hyp}


\begin{lemma}\label{pstructure}(Correction of \cite[Lemma 2.17]{MW3} and \cite[Lemma 2.20]{MW3})
\mbox{}

Suppose that Hypothesis \ref{hypkfix} holds and let $\alpha \in \Omega$ and $p\in \pi(G)$.
Then one of the following holds:

(1) $G_\alpha$ is a $p'$-group.

(2) $G_\alpha$ contains a Sylow $p$-subgroup of $G$.

(3)  The Sylow $p$-subgroups of $G$ have maximal class. If $p=2$, then they must be dihedral or semi-dihedral.

(4) The Sylow $p$-subgroups of $G$ have order at least $p^2$, they have a unique orbit of size $p$, and the remaining orbits are regular.

(5) The Sylow $p$-subgroups of $G$ have order $p^2$.\\

\end{lemma}

\begin{lemma}\label{comp}(Correction of \cite[Lemma 2.21]{MW3})
\mbox{}

Suppose that Hypothesis \ref{hypkfix} holds and let $\alpha \in \Omega$. If $E(G) \neq 1$, then one of the following holds: \\

(1) $E(G) \cap G_\alpha \neq 1$.\\

(2) $E(G)$ is isomorphic to $\Alt_5$ or $\PSL_2(7)$ and $|G_\alpha|=3$.
\end{lemma}

This lemma does not affect 
Theorem 1.3 in \cite{MW3}, while the previous correction Lemma \ref{pstructure} does. More specifically, in Case (i)~(b) of Theorem 1.3 of \cite{MW3}, it is true that $G$ has dihedral or semi-dihedral Sylow $2$-subgroup, but it is possible that $4$ divides the point stabiliser order. Instead of giving more details for the point stabiliser order for a correction, we omitted the faulty statement. 
The cases (ii)~(a) and (b) of Theorem 1.3 of \cite{MW3}  have been summarised into one case with a regular normal subgroup (namely Case (2)~(a) below), and these two changes together give the following:

\begin{thm}\label{main3}(Correction of \cite[Theorem 1.3]{MW3})
\mbox{}

Suppose that $(G, \Omega)$ satisfies Hypothesis \ref{hypkfix}.
Then $G$ has order divisible by $3$
and if $\omega \in \Omega$, then one of the following holds:

\smallskip
(1) $|G_\omega|$ is even and

\begin{itemize}
\item[(a)] $G$ has a normal $2$-complement, or

\item[(b)] $G$ has dihedral or semi-dihedral Sylow $2$-subgroups, or

\item[(c)] $G_\omega$ contains a Sylow $2$-subgroup of $G$ and $G$ has a strongly embedded subgroup, or

\item[(d)] $|G : G_\omega|$ is even, but not divisible by 4 and $G$ has a subgroup of index 2 that has a strongly
embedded subgroup.
\end{itemize}

\smallskip
(2) $|G_\omega|$ is odd and

\begin{itemize}
\item[(a)]  $G$ has a regular normal subgroup, or

\item[(b)] $G$ has a normal subgroup $F$ of index 3 which acts as a Frobenius group on its three orbits, or

\item[(c)] $G$ has a normal subgroup $N$ which acts semi-regularly on $\Omega$  such that $G/N$ is almost simple
and $G_\omega$ is cyclic.
\end{itemize}
\end{thm}

Fortunately, the mistakes that we mentioned do not 
have any consequences for the classification of finite simple groups that act with fixity $3$ or $4$, which is why we do not need any further corrections of main results. 
However, within our analysis in \cite{MW3}, there are after-effects which we will address in this article as well. 

This paper is organised as follows:

After a few preliminary results about specific groups and action with fixity $3$, we prove 
Lemma \ref{pstructure} above and explain how it is used as a replacement for the $2$- and $3$-structure lemmas in \cite{MW3}. We briefly discuss abelian minimal normal subgroups, and then we turn to components of groups that act with fixity $3$.  
Here we prove Lemma \ref{comp} and, again, discuss the consequences for \cite{MW3}.
We revisit all arguments that used our original faulty lemmas, which means that some results will be corrected, re-stated and proven, and others can remain unchanged and will just be checked and discussed briefly. We close by discussing the proof of the original classification theorem for finite simple groups that act with fixity $3$, in Section 5, and by proving Theorem \ref{main3}. 
Sometimes we refer to \texttt{GAP}, see \cite{GAP}, for small cases.

\begin{center}

\textbf{Acknowledgments}

\end{center}

 This article has been written remembering Kay Magaard, with much gratitude and appreciation. We also thank Chris Parker for pointing us to Gerhardt's work. It was challenging to write this corrigendum, which is why we would like to thank the referee for valuable feedback on the clarity and comprehensibility of this article. 


\vspace{1cm}

\section{Preliminaries}

\vspace{0.2cm}

By ``group'' we always mean a finite
group, and we use standard notation for orbits and point stabilisers.
For all $n \in \N$, we denote the cyclic group of order $n$ by
$C_n$.

We begin with some properties of specific groups:

\begin{lemma}\label{exhelp2}
Let $G$ be a finite group, let $E := E(G)$ and let $x \in G$ be of order $3$ such that $|C_E(x)| = 3$. Then $x$ induces an inner automorphism on $E$.
\end{lemma}

\begin{proof}
    Without loss $G = E \<x\>$, and we assume for a contradiction that $x \notin EC_G(E)$. Our hypothesis implies that $|C_G(x)| = 9$ and that, therefore, $x$ does not centralise an involution. According to Theorem 3 in \cite{Ger}, $G$ must be one of the almost simple groups in \cite[Table 1]{Ger}, and since $E \neq G$, we know that $G$ cannot be simple. Now $G$ is isomorphic to $\PGL_3(q)$ or $\PGU_3(q)$, where $q \equiv \epsilon$ modulo $3$ if $\epsilon = 1$ or $\epsilon = -1$, respectively. In both cases, $x$ is contained in an irreducible torus $T$ such that $N_G(T) \cong C_{q^2 + \epsilon q + 1} : C_3$. Since $x$ has order $3$, it follows that $x \in Z(N_G(T))$. In particular, $C_G(x)$ contains a subgroup of order $3 \cdot (q^2 + \epsilon q + 1)$ and therefore $3 = |C_E(x)| \geq q^2 + \epsilon q + 1$. If follows that $(q,\epsilon) = (2,-1)$ and that $G \cong \PGU_3(2)$, which is a soluble group of order $216$. But $1 \neq E(G)$ is not soluble, hence this impossible.
\end{proof}

We fix some notation for the remainder of this article:

Let $\Omega$ be a finite set and suppose that a group $G$ acts on $\Omega$.
Then for all $\Delta
\subseteq \Omega$, all $g \in G$ and all $H \le G$ we let
$\FD(H):=\{\delta \in \Delta \mid \delta^h=\delta$ for all $h \in
H\}$ denote \textbf{the fixed point set of $H$ in $\Delta$} and abbreviate
$\FD(\<g\>)$ by $\FD(g)$.

We collect a few basic technical results from previous papers for
direct reference in this article.

\begin{lemma}\label{normaliser}
Suppose that Hypothesis \ref{hypkfix} holds and let $\alpha \in \Omega$.

(i) If $1 \neq X \le G_{\alpha}$, then $|N_G(X):N_{G_\alpha}(X)| \le |\FO(X)| \le 3$.
In particular, if
$\FO(X)=\{\alpha\}$, then
$N_G(X) \le G_\alpha$.

(ii) If $x \in G_{\alpha}^\#$, then $|C_G(x):C_{G_\alpha}(x)| \le |\FO(x)| \le 3$. In particular, if
$\FO(x)=\{\alpha\}$, then
$C_G(x) \le G_\alpha$.

(iii) $|Z(G)|$ divides $3$.

(iv) If $p \in \pi(G_\alpha)$ and $p \ge 5$, then $G_\alpha$ contains a Sylow $p$-subgroup of $G$.

\end{lemma}

\begin{proof}
These are special cases of Lemma 2.2 in \cite{HWext} which have also, partly, been proven in
 \cite{MW3}.
\end{proof}


\section{The p-structure and abelian minimal normal subgroups}

\begin{lemma}\label{orbits}
Suppose that Hypothesis \ref{hypkfix} holds and let $\alpha \in \Omega$. Moreover let $p\in \pi(G)$,  $P\in \sy_p(G)$ and $n \in \N$ be such that $|P|=p^n$.

If $p\ge 5$, then $\alpha^P$ has size $1$ or $p^n$.
If $p \leq 3$, then $\alpha^P$ has size $1$, $p$, $p^{n-1}$ or $p^n$.
\end{lemma}

\begin{proof}
First suppose that $p\ge 5$.
Then Lemma \ref{normaliser}~(iv) implies that $|\alpha^P|\in \{1,p^n\}$, which is our statement.

Next suppose that $p \in \{2,3\}$ and that $|\alpha^P|\neq 1$. Then $P \nleq G_\alpha$ and hence $|\alpha^P|\ge p$. We recall that all $P$-orbits on $\Omega$ have $p$-power size, we set $\Delta:=\alpha^P$ and we let $m \in \N,l \in \N_0$ be such that
$|\Delta|=|P:P_\alpha|=p^m$ and $|P_\alpha|=p^l$.

\smallskip
\emph{Case 1: $P_\alpha=1$.} Then $|\Delta|=|P|=p^n$, which is one of our possibilities.

\smallskip
\emph{Case 2: $P_\alpha \neq 1$.} Then $l \ge 1$ and the number of fixed points of $P_\alpha$ on $\Delta$ must be divisible by $p$, and also it is bounded by $3$ by Hypothesis \ref{hypkfix}.
Since $P_\alpha$ fixes $\alpha \in \Delta$, it follows that $P_\alpha$ fixes $p$ points. If we let $a \in \N$ denote the number of non-trivial $P_\alpha$-orbits on $\Delta$, then we see that all these $P_\alpha$-orbits have length $p^l$ --otherwise there would be elements in $P_\alpha$ with too many fixed points.
Now $p^m=|\Delta|=p+a \cdot p^l$ and $p^{m-1}=1+a \cdot p^{l-1}$. There are two possibilities for this equation to hold:
$m=1$ and $a=0$ or $m \ge 2$ and $l=1$.
The first case implies that $|\Delta|=p$ and the second case implies that $|\Delta|=p^{n-1}$.
\end{proof}

Next we prove Lemma \ref{pstructure}, which corrects Lemma 2.17 and Lemma 2.20 of \cite{MW3} in the case where $p=2$ and $p=3$, respectively.\\  

\textbf{Proof of Lemma \ref{pstructure}:}

If $p \ge 5$, then Lemma \ref{normaliser}~(iv) implies that (1) or (2) holds.
Therefore, from now on, we suppose that $p \in \{2,3\}$.

If some point stabiliser is a $p'$-group (or contains a Sylow $p$-subgroup), then they all are (or they all do), because $G$ acts transitively on $\Omega$.
Therefore, if we suppose that (1) and (2) do not hold, then all point stabilisers have order divisible by $p$ and index divisible by $p$.
We let $T \in \sy_p(G_\alpha)$ and $P \in \sy_p(G)$ be such that $T \le P$ and we note that $T \neq P$ and $|P| \ge p^2$.

Next we let $\Delta:=\alpha^P$ and we turn to Lemma \ref{orbits}.
Then $|\Delta| \neq 1$ because $P \neq T$, and $|\Delta| \neq |P|$ because $|\Delta|=|P:P_\alpha|=|P:T|$ and $T \neq 1$.

 \smallskip
\emph{Case 1: $|\Delta|=p$.} Then $|P:T|=p$. If $\Delta$ is the unique non-regular orbit of $P$ on $\Omega$, then (4) holds.
Otherwise suppose that $\beta \in \Omega$ is such that $\Delta \neq \beta^P$ and $P_\beta \neq 1$.
We note that $|\beta^P| \neq 1$, because no point stabiliser contains a Sylow $p$-subgroup. Now $|\beta^P| \ge p$ and we recall that $T$ fixes $\alpha$, hence it fixes $\Delta$ point-wise. Since $2 \cdot p >k$, it follows that $T_\beta=1$ and therefore
$T \cap P_\beta=1$. In particular $|T| \cdot |P_\beta|=|P|$, which forces $|P_\beta|=p$.
Lemma \ref{normaliser}~(ii) implies that $|C_P(P_\beta)|=p^2$, hence if $|P| \ge p^3$, then $P$ has
 maximal class. 
 If $p=2$, then $P$ is dihedral, semi-dihedral or quaternion, and the quaternion case does not occur because $T$ must contain the unique involution of $P$. This is (3) (see Satz III.14.23 in \cite{Hupp}).
Otherwise $|P|=p^2$, which is (5).

 \smallskip
\emph{Case 2: $|\Delta|=\frac{1}{p} \cdot |P|$.} Then $|T|=p$ and $|C_P(T)|=p^2$, whence it follows, as in the previous case, that (3) or (5) hold.\\

In the proof of Lemma 2.18 in \cite{MW3}, we can use Lemma \ref{pstructure} instead of
Lemma 2.17 in \cite{MW3}, because if $p=2$, then in Case (2) the only property that is used is the structure of the Sylow $2$-subgroups (namely dihedral or semi-dihedral). Then the main results in \cite{GW} and \cite{ABG} are applied. Therefore, we only have to revisit the results in \cite{MW3} that use Lemma 2.17 directly, not the ones that use Lemma 2.18.

We also noticed that there is a case missing in Lemma 2.19 of \cite{MW3}, because we made a mistake similar to the one in Lemma 2.17. While this does not have any severe consequences, fortunately, we still want to correct the statement:

\begin{lemma}\label{new2.19}(Correction of \cite[Lemma 2.19]{MW3})
\mbox{}

Suppose that Hypothesis \ref{hypkfix} is satisfied and that $|\Omega|$ is odd. Then
one of the following holds:

(1) $G$ has odd order and $3 \in \pi(G)$.

(2) $G$ has a strongly embedded subgroup.

(3) $G$ has a normal $2$-complement. In particular, $G$ is solvable.

(4) $G$ has a normal subgroup $G_0$ of index $2$ that has a strongly embedded subgroup.

(5) $G$ has dihedral or semi-dihedral Sylow $2$-subgroups.

\smallskip
In particular, if $G$ is simple, then $G$ is isomorphic to $\Alt_7$ or to $\M_{11}$  or there exists
a prime power $q$ such that $G$ is isomorphic to $\PSL_2(q),~ \Sz(q),~ \PSU_3(q)$ or
$\PSL_3(q)$.
\end{lemma}

We note that, in the statement of Lemma \ref{new2.19} about simple groups, we just needed to remove the constraint that $q$ is even.

\begin{proof}
We begin as in the original proof and see how Case (1) occurs if the point stabilisers have odd order.
Therefore, we now suppose that $\alpha \in \Omega$ and that $|G_\alpha|$ is even. By hypothesis $G_\alpha$ contains a Sylow $2$-subgroup $S$ of $G$, and by Lemma \ref{orbits} and our main fixity 3 hypothesis there are four possibilities for the orbits of $S$ on
$\Gamma:=\Omega\setminus \{\alpha\}$:

$S$ fixes exactly two points on $\Gamma$, or every element in $S^\#$ is fixed point free on $\Gamma$, or
$S$ has a unique orbit of length $2$ on $\Gamma$, or $S$ has a unique orbit of length $\frac{|S|}{2}$ on $\Gamma$.

The first three cases
are discussed in the original proof in \cite{MW3} and lead to the cases as indicated.
Therefore, we only look at the final case.
We choose $\gamma \in \Gamma$ such that the orbit $\gamma^S$ has size $\frac{|S|}{2}$. Then $|S_\gamma|=2$ and Lemma \ref{normaliser}~(i) forces $|N_S(S_\gamma)|=4$. This implies that $S$ has maximal class, and this leads to Case (5). The new possibilities for finite simple groups, as listed, stem from the main results in \cite{ABG} and \cite{GW}.
\end{proof}

The following result is a preparation for later, but we also improved on the corresponding sub-case of Theorem 2.24 in \cite{MW3} where the minimal normal subgroup is elementary abelian. In particular, we realized that Case (b) of Theorem 2.24 in \cite{MW3} does not occur. The complete corrected version of Theorem 2.24 in \cite{MW3} is found in the next section.

\begin{lemma}\label{solumin}
Suppose that Hypothesis \ref{hypkfix} holds, that $r$ is prime and that $N$ is a minimal normal
subgroup of $G$ that is an elementary abelian $r$-group.
Then one of the following holds:

\begin{enumerate}
\item
[(a)] All Sylow subgroups of all point stabilisers have rank 1.

\item
[(b)] $r=3$, $N$ acts semi-regularly and there are three possibilities:

\begin{itemize}
\item[(i)] $|N| = 3$,
\item[(ii)] $|N| = 9$ and $G_\alpha$ is isomorphic to $C_2 \times C_2$, $D_8$, $SD_{16}$ or $\GL_2(3)$, or
\item[(iii)] $|N| = 27$ and $G_\alpha$ is isomorphic to $C_2 \times C_2$, $D_8$, $SD_{16}$, $\Alt_4$ or $\Sym_4$.
\end{itemize}
\end{enumerate}
\end{lemma}

\begin{proof}
We begin by noting that $N$ is not contained in any point stabiliser because $G$ acts faithfully on $\Omega$.
Suppose that (a) does not hold and let $\alpha \in \Omega$ and $p \in \pi(G_\alpha)$ be such that
$X\le G_\alpha$ is an elementary abelian subgroup of order $p^2$.
We claim that $r \in \{2,3\}$ and then look for the details, hence we assume for a contradiction that
 $r \ge 5$. First Lemma \ref{normaliser}~(iv) implies that the point stabilisers are $r'$-groups, because otherwise $N \le G_\alpha$.
 In particular $p\neq r$ and this means that $X$ acts coprimely on $N$. This gives that
$N=\langle C_N(x) \mid x \in X^\#\rangle \le G_\alpha$, by Lemma \ref{normaliser}~(ii), because $r \ge 5$, which is impossible.

Therefore $r \in \{2,3\}$ as stated above, and we look at each case more closely.

First suppose that $r=2$ and that $N$ acts semi-regularly on $\Omega$.
If $p = 2$, then $NX$ is contained in a Sylow $2$-subgroup $S$ of $G$.
Since $|N| \ge 4$ and $|X|=4$, we see that $|\alpha^S|$ is none of the numbers from Lemma \ref{orbits}, and therefore this case does not occur.
Hence $p \neq 2$ and the action of $X$ on $N$ is coprime. In particular $N = \< C_N(x) \mid x \in X^\# \>$. Combining Lemma \ref{normaliser}~(ii) and the fact that $N$ acts semi-regularly, it follows for all $x \in X^\#$ that $|C_N(x)| \leq 3$. But $|N| \ge 4$, hence we find $x_1,x_2 \in X^\#$ such that $C_N(x_1)$ and  $C_N(x_2)$ are distinct and both have order 2. We recall that $X$ is abelian, which implies that $x_2$ normalises $C_N(x_1)$, and then it must centralise it. But $C_N(x_1) \neq C_N(x_2)$, which gives a contradiction. Consequently, if $r=2$, then $N_\alpha \neq 1$.
 
If $1 \neq N_\alpha$, then the regular action of $G$ gives $1 \neq N_\alpha \neq N$ and Lemma \ref{normaliser} (i) yields $|N:N_\alpha| = 2$. This implies that every element in $G$ has an even number of fixed points on $\Omega$, contrary to our fixity $3$ hypothesis. Hence we conclude that $r \neq 2$.

Next suppose that $r=3$ and that $N$ acts semi-regularly.
If $|N| = 3$ , then (b)~(i) holds. Now suppose that $|N| \geq 9$. If $p = 3$ and $NX \leq P \in \sy_3(G)$, then we argue as above: $|P_\alpha| \geq 9$ and $|P:P_\alpha| \geq 9$, and this gives an impossible orbit length for $\alpha^P$, by Lemma \ref{orbits}. Hence $X$ acts coprimely on $N$ and $N = \< C_N(x) \mid x \in X^\# \>$.
We recall that $|C_N(x)| \leq 3$ for all $x \in X^\#$ by Lemma \ref{normaliser}~(ii), because $N_\alpha=1$.
Then we let $x_1,x_2 \in X^\#$ be such that $C_N(x_1)$ and $C_N(x_2)$ are distinct and have order 3. Since $X$ is abelian, we see that $x_2$ normalises $C_N(x_1)$, but does not centralise it, and therefore $x_2$ induces a non-trivial automorphisms on $C_N(x_1) \cong C_3$. In particular $p$ divides $r - 1 = 2$ and thus $p = 2$ and $|X^\#| = 3$. Since $N = \<C_N(x) \mid x \in X^\#\>$, this forces $|N| \in \{9,27\}$ and Lemma \ref{normaliser}~(ii) implies that $G_\alpha$ induces non-trivial automorphisms on $N$.
With \texttt{GAP} (see \cite{GAP}) we find all subgroups $H \leq \au(N) \cong \GL_2(3)$ or $\GL_3(3)$, respectively, that satisfy the property $|C_N(h)| \leq 3$ for all $h \in H^\#$. According to the \texttt{GAP} calculations, the only cases where some Sylow subgroup of $H$ does not have rank $1$ are listed in the statement.

Finally, if $N_\alpha \neq 1$, then $N_\alpha$ fixes exactly three points, and Lemma \ref{normaliser}~(ii) implies that $|N:N_\alpha| = 3$ and that all $N$-orbits have size $3$. If $\beta \in \Omega \setminus \alpha^N$, then the fixity $3$ hypothesis yields that $N_\alpha \cap N_\beta = 1$. We also know that $|N_\alpha| = |N_\beta|$ and $|N:N_\alpha| = 3$, which forces $|N| = 9$ and $|N_\alpha| = 3$. Moreover, there are at most four $N$-orbits in total because $N$ has only four subgroups of order $3$. We conclude that $|\Omega| \leq 12$ and we can check the possibilities with \texttt{GAP}. The only case that occurs is $G \cong (C_3 \times C_3) : C_4$ and $G_\alpha \cong C_6$. This is included in Case (a).
\end{proof}

In the next section, we will revisit components and prove Lemma \ref{comp}, and then we prove a new version of Theorem 2.24 of \cite{MW3}.

\section{Components and minimal normal subgroups}

We begin by proving Lemma \ref{comp}, which corrects Lemma 2.21 of \cite{MW3}:\\

\textbf{Proof of Lemma \ref{comp}:}

Suppose that $E(G) \neq 1$ and that $E(G) \cap G_\alpha = 1$. Let $E$ be a component of $G$ and let $x \in G_\alpha$ be of prime order. As in the proof of Lemma 2.21 in \cite{MW3}, we can show that $x \in N_G(E)$, that $o(x) = 3$ and that $|C_E(x)| = 3$.
In particular, $G_\alpha$ must be a $3$-group and it normalises $E$.
The mistake happens towards the very end, when we refer to \cite{CC}. Therefore, we analyse the situation again now, using our Lemma \ref{exhelp2}.
We recall that $x \in G_\alpha$, therefore $x \notin E$
and $E$ is the unique component in $H:=E\langle x \rangle$.

Let $Z := Z(E)$ and $\ov{H} := H/Z$. Furthermore, let $C \leq H$ be the full preimage of $C_{\ov{E}}(\ov{x})$ in $E$. For all $c \in C$ we define $c^\varphi := [x,c]$. Then $\varphi$ defines a map from $C$ to $Z$, and for all $c_1,c_2 \in C$, we have that $(c_1 c_2)^\varphi = [x,c_1 c_2] = [x,c_2] \cdot [x,c_1]^{c_2} = [x,c_1] \cdot [x,c_2] = c_1^\varphi \cdot c_2^\varphi$ because $[x,c_1],[x,c_2] \in Z = Z(E)$. Hence $\varphi$ is a homomorphism with kernel $C_C(x) = C_E(x)$ of order $3$. It follows that $|C| \leq 3 |Z|$ and therefore $|C_{\ov{E}}(\ov{x})| \leq 3$.

Since $\ov{E}$ is non-abelian simple, it is not nilpotent, and therefore $C_{\ov{E}}(\ov{x})$ is non-trivial (due to a result of Thompson's, see 9.5.1 in \cite{KS}).
Assume for a contradiction that $|C_{\ov{E}}(\ov{x})| = 2$. Then $N := O_2(C) \neq 1$ and $|N:Z \cap N| = 2$. Now the coprime action of $x$ on $N$ yields that $|C_N(x)| \geq 2$ (see for example 8.2.2 in \cite{KS}), in contradiction to the fact that $|C_E(x)| = 3$ is odd. Hence $|C_{\ov{E}}(\ov{x})| = 3$ and Lemma \ref{exhelp2} gives that $\ov x$ induces an inner automorphism on $\ov E$. Now $\ov{E}$ is isomorphic to $\Alt_5$ or $\PSL_2(7)$ by the main result in \cite{FT1962}.
Both groups have a Schur multiplier of order $2$. Thus, if $Z(E) \neq 1$, then $Z(E) \leq C_E(x)$, which is impossible because $|C_E(x)| = 3$. 
We recall that $G_\alpha$ is a $3$-group that normalises $E$, from the first paragraph, and we also recall that $E \cong \Alt_5$ or $\PSL_2(7)$. Both groups have outer automorphism groups of order 2, and $E$ cannot be centralised by any non-trivial subgroup of $G_\alpha$. Then it follows that $|G_\alpha|=3$.

\begin{rem}\label{ex4.1}
At the end of the proof of the previous lemma, we can say a bit more:
There must be a subgroup $Y$ of $G$ of order $3$ such that $G=E \times Y$ and such that for each $\alpha \in \Omega$, there is a subgroup $T$ of $E$ of order $3$ such that $G_\alpha \le TY$.

In fact, every group of structure $\Alt(5) \times C_3$ or $\PSL_2(7) \times C_3$ exhibits an example for a fixity 3 action in exactly this way.  
\end{rem}

Lemma 2.22 in \cite{MW3} is still correct, but we repeat it here and prove it again, because it originally relied on a faulty result.

\begin{lemma}\label{onecomp}(\cite[Lemma 2.22]{MW3})
\mbox{}

Suppose that Hypothesis \ref{hypkfix} holds and that $E(G) \neq 1$. Then $G$ has a unique component.
\end{lemma}

\begin{proof}
If $E(G) \cap G_\alpha = 1$, then Lemma \ref{comp} implies that $G$ has a unique component.
Otherwise we can argue as in the original Lemma 2.22 of\cite{MW3}.
\end{proof}

One additional note on this:
If we assume for a contradiction that
$G$ has several components, then the arguments in
the proof of Lemma 2.22 in \cite{MW3} and the main result of \cite{FT1962} show that $G$ has exactly two components $E$ and $L$ and that both of them are isomorphic to $\Alt_5$ or to $\PSL_2(q)$.
In both cases there is an involution in $E(G)$ inverting
a subgroup of order 3 in $E$ and $L$, respectively, which leads to a contradiction.
We mention this here because our proof in \cite{MW3} only discusses the case $\PSL_2(7)$ explicitly.\\

Lemma 2.23 in \cite{MW3} needs an additional case now, following Lemma \ref{comp}:

\begin{lemma}\label{compcases}(Correction of \cite[Lemma 2.23]{MW3})
\mbox{}

Suppose that Hypothesis \ref{hypkfix} holds and that $E$ is a component of $G$. Then one of the following holds:

(1) $E \cap G_\alpha = 1$, $E \cong \Alt_5$ or $\PSL_2(7)$ and $G_\alpha$ is a $3$-group.

(2) There exists a power $q$ of $2$ such that $E \cong \PSL_2(q)$, $|G : E|$ is prime and every element from $G \setminus E$ induces a field automorphism on $E$. For all $\alpha \in \Omega$, we have that
$|G_\alpha|= q \cdot (q-1) \cdot |G : E|$ and, moreover, $E_\alpha$ does not contain any elements that fix three points.

(3) There is some $\alpha \in \Omega$ such that $(E, \alpha^E)$ satisfies Hypothesis \ref{hypkfix}.
\end{lemma}

\begin{proof}
If $E(G) \cap G_\alpha = 1$, then Lemma \ref{comp} gives (1). Otherwise, the arguments from the proof of Lemma 2.23 of \cite{MW3} apply without changes.
\end{proof}

In Theorem 2.24 of \cite{MW3}
we need to add the possibility of a component isomorphic to $\PSL_2(7)$ or $\Alt_5$ with trivial intersection with the point stabilisers. We also organise the cases differently, based on our previous alterations. All the work has been done in previous lemmas, including necessary corrections.

\begin{thm}\label{min}(Correction of \cite[Theorem 2.24]{MW3}
\mbox{}

Suppose that Hypothesis \ref{hypkfix} holds and that $N$ is a minimal normal
subgroup of $G$. Let $\alpha \in \Omega$. Then one of the following holds:

\begin{enumerate}
\item
[(a)] All Sylow subgroups of all point stabilisers have rank 1.

\item
[(b)] $N$ is a 3-group and it acts semi-regularly on $\Omega$. One of the following holds:

\begin{itemize}
\item[(i)] $|N| = 3$,
\item[(ii)] $|N| = 9$ and $G_\alpha$ is isomorphic to $C_2 \times C_2$, $D_8$, $SD_{16}$ or $\GL_2(3)$, or
\item[(iii)] $|N| = 27$ and $G_\alpha$ is isomorphic to $C_2 \times C_2$, $D_8$, $SD_{16}$, $\Alt_4$ or $\Sym_4$.
\end{itemize}

\item
[(c)] $N = E(G)$ and $(N, \alpha^N)$ satisfies Hypothesis \ref{hypkfix}.

\item[(d)] $N=E(G)$, $N \cap G_\alpha \neq 1$ and
$N \cong \PSL_2(q)$, where $q$ is a $2$-power.

\item[(e)] $N = E(G)$, $N \cap G_\alpha = 1$, $N \cong \Alt_5$ or $\PSL_2(7)$ and $G_\alpha$ is a $3$-group.
\end{enumerate}
\end{thm}

\begin{proof}
If $N$ is soluble, then we refer to Lemma \ref{solumin}.
Otherwise, $N$ is a product of components and
then $N$ is the unique component of $G$, by Lemma \ref{onecomp}. Then Lemma \ref{compcases} gives the remaining possibilities exactly as stated.
\end{proof}

We now follow the applications of Lemma 2.23 and Theorem 2.24 of \cite{MW3} and we check where adjustments need to be made.

\begin{lemma}(\cite[Lemma 2.25]{MW3})
\mbox{}

Suppose that Hypothesis \ref{hypkfix} holds. Then $3 \in \pi(G)$.
\end{lemma}

\underline{Comments on the proof:}

As in the original proof, we let $\alpha \in \Omega$ and we see that $G_\alpha$ has even order. Then we turn to Claim (1). Under the assumption there, if we have a component, then there is a unique one and this is the only place in the proof where Lemma 2.23 is used.
According to Lemma \ref{compcases}, we have three possibilities, and if $3 \notin \pi(G)$, then (3) must hold and
the unique component is a Suzuki group.
We assume for a contradiction that $q$ is a power of $2$ such that $E \cong \,\Sz(q)$ acts on $\Delta:=\alpha^E$ with fixity 3.
Since $3 \notin \pi(E)$, we can argue with the subgroup structure as we do in Lemma 3.12 of \cite{MW2}, and we obtain the possibilities for point stabilisers. But these lead to fixity 2 actions, not fixity 3.
Now it follows that $E(G)=1$, and then we recall that $O_2(G)=1=O_3(G)$, which means that $F^*(G)=F(G)$ and that Case (a) of Theorem \ref{min} holds. This gives Claim (1).

For Claim (2), there is a reference to Lemma 2.17 of \cite{MW3} (2-structure), and we use Lemma \ref{pstructure} as a replacement.
Since $G_\alpha$ has even order, we first consider Case (2) of Lemma \ref{pstructure}, and the original arguments apply.
The treatment of the next case can also stay unchanged, because it only uses the fact that the Sylow 2-subgroup are dihedral or semi-dihedral, as in our Case (3) of Lemma \ref{pstructure}. Instead of Lemma 2.17~(3) of \cite{MW3}, we look at the cases (4) and (5) of Lemma \ref{pstructure} and give explicit arguments here.
If Case (4) holds, then a Sylow $2$-subgroup $S$ of $G$ has order at least $4$, it has a unique orbit of size 2 and regular orbits otherwise. We choose $S$ such that $1 \neq S_\alpha \in \sy_2(G_\alpha)$ and we note that $|S:S_\alpha|=2$. Let $\beta \in \Omega$ be such that $\{\alpha,\beta\}$ is the unique $S$-orbit of size $2$. Now $S_\alpha=S_\beta$ and we may suppose that $S$ is not cyclic, because otherwise $G$ has a subgroup of index $2$ by Burnside's Theorem, as claimed.
But then we argue as in the original proof and find an odd permutation in $G$ that interchanges $\alpha$ and $\beta$. This means that the intersection of $G$ with the alternating group on $\Omega$ has index $2$ in $G$.
If Case (5) of Lemma \ref{pstructure} holds, then a Sylow $2$-subgroup $S$ of $G$ has order $4$. If $S$ is cyclic, then our claim holds by Burnside's Theorem again, and if $S$ is elementary abelian, then Frobenius' Theorem gives a normal $2$-complement and hence a subgroup of index $2$ of $G$. This finishes the proof of Claim (2), and
the remainder of the original proof in \cite{MW3} can stay unchanged.

This is sufficient for the analysis of simple groups that act with fixity $3$, which is why we only discuss a few more results from Section 3 and 4 of \cite{MW3}.\\

\begin{lemma}(\cite[Lemma 3.4]{MW3})
\mbox{}

Suppose that $G$ is almost simple, but not simple and suppose further
that $F^*(G) \cong \A_6$. Then there does not exist a set $\Omega$ such that $(G, \Omega)$ satisfies Hypothesis \ref{hypkfix}.
\end{lemma}

\underline{Comments on the proof:}
In this case $F^*(G)=E(G)$ is the unique component and Lemma \ref{compcases}
applies. The only possibility is Case (3), which means that we can finish the argument as we originally do in \cite{MW3}.\\

\begin{cor}(\cite[Corollary 3.6]{MW3})
\mbox{}

Suppose that $G \cong \Sym_7$. Then there is no set $\Omega$ such that $(G, \Omega)$ satisfies Hypothesis \ref{hypkfix}.
\end{cor}

\underline{Comments on the proof:} If we let $E$ denote the unique component, then it is isomorphic to $\Alt_7$.
As in the previous result,  Case (3) of Lemma \ref{compcases} is the only possibility, and the arguments of the original proof in \cite{MW3} can remain unchanged.\\

\begin{lemma}\label{alt3}(Stronger version of \cite[Lemma 3.8]{MW3})
\mbox{}

Suppose that Hypothesis \ref{hypkfix} holds and that $G$ is an alternating group of degree $n \ge 8$.
Then $n=8$, $|\Omega|=2880$ and $G$ acts on $\Omega$ as on the set of cosets of a Sylow $7$-subgroup.
\end{lemma}

\begin{proof}
The Sylow $2$-subgroups of $G$ have order at least $2^6$ and they are neither dihedral nor semi-dihedral, and $|\Omega| \ge 8$ because $G$ acts faithfully.
Therefore, if $\alpha \in \Omega$, then Case (1), (2) or (4) of Lemma \ref{pstructure} holds.

Calculations in \texttt{GAP} (see \cite{GAP} and appendix) show that the case $n=8$ gives a unique example, with point stabilisers of order 7. Now we prove that there are no examples if $n \ge 9$. We assume otherwise, and
we keep it short because the arguments in \cite{MW3} are correct as long as we are careful with the $2$-structure.
First we note that $G_\alpha$ does not contain a $3$-cycle, because otherwise, with Lemma \ref{normaliser}(ii),
$G_\alpha$ contains a subgroup isomorphic to $\A_6$, and then several applications of the same result show that $G_\alpha$ cannot be a proper subgroup of $G$.

Case (1): $G_\alpha$ has odd order.

Let $x \in G_\alpha$ be of prime order $p$ and such that $x$ fixes three points.
Then the subgroup structure of $G$ and Lemma \ref{normaliser}(i) imply that $p=3$, and the same lemma yields that $x$ is not a product of two 3-cycles. Then it must be a product of at least three 3-cycles. If $n \in \{9,10,11\}$, this means that $x$ is $3$-central and that $G_\alpha$ contains a 3-cycle, by Lemma \ref{normaliser}(ii). This is impossible.
Therefore $n \ge 12$. Again by Lemma \ref{normaliser}(ii), $C_G(x)$ cannot have order divisible by 4, and this implies that $x$ cannot be a product of four or more 3-cycles. The only possibility left is $n=12$, and $C_G(x)$ contains a subgroup of structure $((3^3:\Sym_3) \times \Sym_3) \cap \A_{12}$. By Lemma \ref{normaliser}(ii), it follows that $G_\alpha$ contains a 3-cycle or a product of two 3-cycles, and this gives a contradiction.

Case (2) and (4): These cases have in common that $G_\alpha$ has even order.

Since $n \ge 9$, it follows with Lemma \ref{normaliser}~(ii) that $G_\alpha$ contains a double transposition and then that it contains a subgroup of index $2$ or $3$ of its centraliser, which is impossible.
\end{proof}

\begin{cor}(\cite[Corollary 3.9]{MW3})
\mbox{}

Suppose that $G \cong \Sym_8$. Then there is no set $\Omega$ such that $(G,\Omega)$ satisfies Hypothesis \ref{hypkfix}.
\end{cor}

\underline{Comments on the proof:} 

Again, Case (3) of Lemma \ref{compcases} holds, and all arguments can stay as in \cite{MW3}.\\

\begin{lemma}(Symmetric group cases of \cite[Lemma 3.10]{MW3} and \cite[Lemma 3.11]{MW3})
\mbox{}

Suppose that Hypothesis \ref{hypkfix} holds and that $G \cong \Sym_9$ or $G \cong \Sym_{10}$.
Then there is no set $\Omega$ such that $(G,\Omega)$ satisfies Hypothesis \ref{hypkfix}.
\end{lemma}

\underline{Comments on the proof:}
In both cases of the hypothesis, Case (3) of Lemma \ref{compcases} applies and leads to a contradiction just like in \cite{MW3}.\\

\begin{thm}(Correction of \cite[Theorem 3.14]{MW3})
\mbox{}

Let $n \in \N$ and suppose that $G$ is isomorphic to $\A_n$ or to $\Sym_n$. If $\Omega$
is a set such that $(G, \Omega)$ 
satisfies Hypothesis \ref{hypkfix}, then 
$n \in  \{5, 6, 7, 8\}$ and one of
the following holds:

\begin{enumerate}
    \item[(1)] $n = 5, G \cong \Alt_5$ or $\Sym_5$, $|\Omega|= 15$ and the action of $G$ is as on the set of cosets of
a Sylow $2$-subgroup.

\item[(2)] $n = 5, G \cong \Sym_5, |\Omega| = 5$ and $G$ acts naturally.

\item[(3)] $n = 6, G \cong \A_6, |\Omega|= 6$ and $G$ acts naturally.

\item[(4)] $n = 6, G \cong \A_6, |\Omega|= 15$ and $G$ acts as on the set of cosets of a subgroup of
order $24$.

\item[(5)] $n= 7, G \cong \A_7, |\Omega|= 15$ and the action of $G$ is as on the set of cosets of
a subgroup isomorphic to $\PSL_2(7)$.

\item[(6)] $n = 7, G \cong \A_7, |\Omega| = 360$ and the action of $G$ is as on the set of cosets of
a Sylow $7$-subgroup.

\item[(7)] $n = 8, G \cong \A_8, |\Omega| = 2880$ and the action of $G$ is as on the set of cosets of
a Sylow $7$-subgroup.
\end{enumerate}
\end{thm}

\underline{Comments on the proof:} 
The theorem coincides with Theorem 3.14 of \cite{MW3} except for an additional case in (1), where $G \cong \Sym_5$. This case had been neglected in the statement of the theorem although it had been analysed correctly in Lemma 3.2  of \cite{MW3}.\\

Following potential consequences of the mistakes in Lemmas 2.17 and 2.20 of \cite{MW3}, we state a special Hypothesis from \cite{MW3}:

\begin{hyp}\label{hypLie}

Suppose that $(G,\Omega)$ satisfies Hypothesis \ref{hypkfix} and that $G$ is simple of Lie type, but none of the groups $\PSL_2(q), \Sz(q)$ or $\PSU_3(q)$ where $q$
is even, or $\PSL_2(q), \PSU_3(q)$ or $\PSL_3(q)$ where $q$ is odd. Suppose further that the point stabilisers have odd order.   
\end{hyp}

\begin{lemma}(\cite[Lemma 4.9]{MW3}, rephrased, with a typo corrected)
\mbox{}

If Hypothesis \ref{hypLie} holds, then $G \not \cong \PSp_4(3)$.
\end{lemma}

\underline{Comments on the proof:} 

The typo that we corrected is that it should be $\PSp_4(3)$ instead of $\Sp_4(3)$, which is what we prove and what we use later. 
If we assume for a contradiction that $G \cong \PSp_4(3)$ and let $\omega \in \Omega$, then we can follow the initial argument for Lemma 4.9 of \cite{MW3}, and we see that $G_\omega$ is a $3$-group, because it has odd order. In particular we find an element $x\in G_\omega$ of order $3$. 
We can check with \texttt{GAP} that $G$ has four conjugacy classes of elements of order $3$, 
and in each case $|C_G(x)|$ is divisible by $27$. 
We also know that the Sylow $3$-subgroups of $G$ have structure $C_3 \wr C_3$, and then it follows that $G_\omega$ contains a $3$-central element. Its centraliser has order $648$, which is divisible by $8$, and this contradicts Lemma \ref{normaliser}~(ii).\\

\begin{lemma}(\cite[Lemma 4.11]{MW3})
\mbox{}

If Hypothesis \ref{hypLie} holds, then all non-identity elements in the 
point stabilisers are regular semi-simple elements.
\end{lemma}

\underline{Comments on the proof:}

We just look at the places in the original proof where Lemma 2.20 of \cite{MW3} was originally used. If $\omega \in \Omega$ and $G_\omega$ contains a $3$-element, then the lemma is used for the argument that $G_\omega$ contains a subgroup of index at most $3$ of a Sylow $3$-subgroup of $G$.
We check our Lemma \ref{pstructure}: Case (1) cannot hold because $3 \in \pi(G_\omega)$, Case (3) cannot hold because the Sylow 3-subgroups of $G$ do not have maximal class, and in all other cases we see that $G_\omega$ contains a subgroup of index at most $3$ of  a Sylow $3$-subgroup of $G$.
Therefore, our new Lemma \ref{pstructure} can replace Lemma 2.20 of \cite{MW3} in this situation, and the rest of the arguments can remain unchanged.\\

We do not re-state Lemma 4.19 of \cite{MW3}, where almost simple groups with socle $\PSL_3(q)$ are considered.
All cases where $q\le 4$ are investigated using \texttt{GAP}, which also takes care of the exceptional isomorphism $\PSL_3(2) \cong \PSL_2(7)$. Therefore, we do not need to correct anything.


\section{Discussion of Theorem 1.1 of \cite{MW3} and the proof of Theorem \ref{main3}}

The main result in \cite{MW3} for simple groups is correct. We re-state it here and briefly look at its proof, in order to check that Lemma \ref{pstructure} is an adequate replacement for Lemma 2.17 of \cite{MW3}.\\

\begin{thm}(\cite[Theorem 1.1]{MW3})
\mbox{}

Suppose that Hypothesis \ref{hypkfix} holds, that $G$ is simple and that $\omega \in \Omega$. 

Then one of the following holds:

\begin{enumerate}
    
\item[(i)] $G_\omega$ is not cyclic and one of the following is true:

\begin{enumerate}
\item[(a)] $G \cong \A_5, |\Omega| = 15$ and $G_\omega \in \sy_2(G)$.

\item[(b)] $G \cong \A_6, |\Omega| \in \{6,15\}$ and $G_\omega$ is isomorphic to $\A_5$ or $\Sym_4$, respectively.

\item[(c)] $G \cong \PSL_2(7), |\Omega|= 7$ and $G_\omega \cong \Sym_4$.

\item[(d)] $G \cong \A_7, |\Omega|=15$ and $G_\omega \cong \PSL_2(7)$.

\item[(e)] $G \cong \PSL_2(11), |\Omega|= 11$ and $G_\omega \cong \A_5$.

\item[(f)] $G \cong \M_{11}, |\Omega|=11$ and $G_\omega \cong \M_{10} \cong \A_6^{\cdot}2$.
\end{enumerate}

\item[(ii)] $G_\omega$ is cyclic of order prime to $6$ and one of the following is true:

\begin{enumerate}

\item[(a)] $G \cong \PSL_3(q)$ and $|G_\omega|= \frac{q^2 + q + 1}{(3, q -1)}$.

\item[(b)] $G \cong \PSU_3(q)$ and $|G_\omega| = \frac{q^2 - q + 1}{(3, q + 1)}$.

\item[(c)] $G \cong \PSL_4(3), |\Omega| = 2^7 \cdot 3^6 \cdot 5$ and $|G_\omega|= 13$.

\item[(d)] $G \cong \PSU_4(3), |\Omega| = 2^7 \cdot 3^6 \cdot 5$ and $|G_\omega|= 7$.

\item[(e)] $G \cong \PSL_4(5), |\Omega| = 2^7 \cdot 3^2 \cdot 5^6 \cdot 13$ and $|G_\omega| = 31$.

\item[(f)] $G \cong \A_7, |\Omega| = 360$ and $|G_\omega| = 7$.

\item[(g)] $G \cong \A_8, |\Omega| = 2880$ and $|G_\omega| = 7$.

\item[(h)] $G \cong \M_{22}, |\Omega| = 2^7 \cdot 3^2 \cdot 5 \cdot 11$ and $|G_\omega| = 7$.
\end{enumerate}
\end{enumerate}
\end{thm}

\underline{Comments on the proof:}

Using the CFSG, we first consider alternating groups and hence Theorem 3.14 of \cite{MW3}.
The actual possibilities can be checked with \texttt{GAP} (\cite{GAP}), and for the proof that no other alternating groups
give rise to examples, we use Lemma \ref{alt3}.

For the groups of Lie type, we turn to Theorem 4.17 of \cite{MW3} which gives a summary of all fixity $3$ examples that come from Lie type groups. Some lemmas leading up to it use the $2$-structure (more specifically Lemma 2.18 of \cite{MW3}), and we checked that all arguments are valid.

Finally, for the sporadic simple groups, Theorem 5.8 of \cite{MW3} gives that $\M_{11}$ and $\M_{22}$ are the only examples.
There are two lemmas leading up to this, using the $2$-structure, but again we checked that our Lemma \ref{pstructure} is an adequate replacement or that Lemma 2.18 of \cite{MW3} is used, which we have already checked and which is correct.\\

The example $\A_6$ with point stabilisers isomorphic to $\A_5$ also indicates that there is indeed a case missing in Theorem 1.3 of \cite{MW3}, and that this case corresponds to the missing case in Lemma 2.17~(2) of \cite{MW3}. We will discuss a proof of our main result about groups that act with fixity $3$ now, and we prepare this by revisiting Propositions 6.2 and 6.3 from \cite{MW3}.

\begin{prop}\label{6.2}(Correction of \cite[Proposition 6,2]{MW3})
\mbox{}

Let $\Omega$ be a set such that $(G, \Omega)$ satisfies Hypothesis \ref{hypkfix}, let $\omega \in \Omega$ and suppose that $G_\omega$ has
even order. Then one of the following is true:

(1) $G$ has a normal $2$-complement.

(2) $G$ has dihedral or semi-dihedral Sylow 2-subgroups and $G_\omega$ has twice odd order or
twice odd index in $G$.

(3) $G_\omega$ contains a Sylow $2$-subgroup $S$ of $G$ and $G$ has a strongly embedded subgroup.

(4) $G_\omega$ has twice odd index in $G$ and $G$ has a subgroup of index $2$ that has a strongly
embedded subgroup.
\end{prop}

\begin{proof}
By hypothesis we have one of the cases (2) -- (5) of Lemma \ref{pstructure}.
Case (2) implies that $|\Omega|$ is odd and then Lemma \ref{new2.19} is applicable. Since $G$ has even order and $G_\omega$ contains a Sylow subgroup, this leads to (1) or (3) of our statement.
If Case (2) of Lemma \ref{pstructure} does not hold, then we let $S \in \sy_2(G)$ and we turn to Lemma \ref{orbits}. It gives that $|S:S_\omega|=2$ or $|S_\omega|=2$.
If $|S|=4$, then $S$ is cyclic, whence (1) holds by Burnside's Theorem, or $S$ is a fours group, a special case of a dihedral group, and then (2) holds.
Now suppose that Lemma \ref{pstructure}~(4) holds. Again, if $S$ is cyclic, then (1) holds by Burnside's Theorem. Otherwise the orbit structure implies that $S$ contains elements that act on $\Omega$ as an odd permutation. This gives a subgroup of $G$ of index 2, and then the original arguments in \cite{MW3} apply, leading to Case (4) of our proposition.

Finally, if Lemma \ref{pstructure}~(3) holds, then the combination with our statement above (namely that $|S:S_\omega|=2$ or $|S_\omega|=2$) gives (2).
\end{proof}

In Case (2) there are still several cases where the point stabilisers have cyclic Sylow $2$-subgroups and hence a normal $2$-complement, but this is not the only possibility. This can be seen in $\Alt_6$, which has a fixity 3 action with point stabilisers isomorphic to $\Alt_5$.\\

Before we revisit Proposition 6.3 of 
\cite{MW3} and its proof, we discuss Lemma 6.6. 
This was originally meant to add detail, and its proof was independent of the main results and of Proposition 6.3. We discuss it here for several reasons: It helps to have this information in the case distinction that we use for our proof of a new version of Proposition 6.3 of \cite{MW3}, and treating this case first avoids repetitive arguments later. Moreover, the original statement is not correct because of the possibility of components that act semi-regularly (as we have seen in Lemma \ref{comp}). As indicated above, this did not affect the main results of \cite{MW3}.\\

\begin{lemma}\label{new6.6}(Correction of \cite[Lemma 6.6]{MW3})
\mbox{}

Suppose that Hypothesis \ref{hypkfix} holds and let $\omega \in \Omega$. Suppose that $G_\omega$ 
is a $3$-group and that $|\FO(G_\omega)|= 3$. If G has a regular normal subgroup $N$, then one of the following holds:

(1) $G_\omega$ is cyclic, $N$ is soluble and $G=O_{3,3'}(G)G_\omega$ or

(2) $G/O_{3'}(G)$ is almost simple and the unique component is isomorphic to $\Alt_5$ or to $\PSL_2(7)$.

(3) $N/O_{3'}(G) \cong \Sym_3$ and a Sylow $3$-subgroup of $G/O_{3'}(G)$ has index 2 in $G/O_{3'}(G)$.

\end{lemma}

Case (1) is the original statement from Lemma 6.6 of \cite{MW3}.
The other cases were overlooked for the same reasons that certain types of components were overlooked (see Lemma \ref{comp}).

\begin{proof}
Let $H:=G_\omega$ and let $H \le P \in \sy_3(G)$. We first follow the original proof: $G=NH=NP$, $|C_N(H)|=3$ and $H \neq P$ because $3 \in \pi(N)$. Then we use Lemma \ref{pstructure} and see that only the cases (3) -- (5) can apply. First suppose that $P$ does not have maximal class.   
If $P$ has order at least $p^2$ and a unique non-regular orbit, then this non-regular orbit must be $\omega^P$, and it consists of the three fixed points of $H$. In particular $|P:H|=3$. Now let $P_0:=P \cap N$. Since $N$ acts regularly, it follows that $|P_0|=3$ and in particular $N$ has Sylow $3$-subgroups of order $3$.

Now we recall that $G=NP$, which implies that $O_{3'}(G)=O_{3'}(N)$, and we let $\bar G:=G/O_{3'}(G)$. We note that $F(\bar G)$ is a $3$-group that contains $F(\bar N)$, because $\bar N \unlhd \bar G$. Moreover $E(\bar N) \le E(\bar G)$. Since $|G:N|$ is a $3$-power, we see that, conversely, $E(\bar G)\le \bar N$. Before we go into the next case distinction, we also recall that $\bar N$ has Sylow $3$-subgroups of order $3$.

First suppose that $E(\bar G) \neq 1$.
Assume for a contradiction that $\bar E$ is a non-simple component of $\bar N$. Then it has a central Sylow $3$-subgroup of order $3$, because $O_{3'}(\bar N)=1$, and this is impossible. 
Thus all components of $\bar G$, and hence of $\bar N$, are simple.
If there is any component isomorphic to a Suzuki group, then the product of the "Suzuki components" is a normal $3'$-subgroup of $\bar N$, which must be trivial. Therefore, all components of $\bar N$ are simple and have order divisible by $3$, which implies that there can only be a unique component 
$\bar E$ and that $O_3(\bar N)=F(\bar N)=1$. Now $F^*(\bar N)=\bar E$ and $C_{\bar N}(\bar E)=1$. We deduce that $\bar N$ is almost simple.  We recall that $H$ and $O_{3'}(N)$ have coprime orders and that $|C_N(H)|=3$, in fact $C_N(H)=P_0$. From there it follows first that $C_{\bar N}(\bar H)=\overline{C_N(H)}=\bar{P_0} \le \bar E$ and then that
$\bar G$ is almost simple. 
For each element $x$ of $H$ such that $\bar x$ has order $3$, we can apply Lemma \ref{exhelp2}, and it implies that $\bar H$ induces inner automorphisms on $\bar E$. 
Then the main result in \cite{FT1962} yields that $\bar E \cong \Alt_5$ or $\PSL_2(7)$ as stated in Case (2).

If $F^*(\bar G)=F(\bar G)$, then it is a $3$-group. Since $\bar N \unlhd \bar G$, this forces $F^*(\bar N)=F(\bar N)=
\bar P_0$. Consequently $\bar P_0$ 
contains its centraliser in $\bar N$.
We apply the main result in \cite{FT1962} and recall that $F(\bar N)$ is a Sylow $3$-subgroup of $\bar N$, which excludes most cases.
The only possibility is that $\bar N \cong \Sym_3$ or $\Alt_3$.
If $\bar N \cong \Alt_3$, then $G=O_{3'}(G) \cdot P$ or, in other words, $G$ has a normal $3$-complement. Here we can argue as in \cite{MW3} and we see that (1) holds.
If $\bar N \cong \Sym_3$, then the fact that $G=NP$ implies that $\bar G$ has Sylow $2$-subgroups of order 2, and hence $\bar P$ is a normal subgroup of $\bar G$ of index 2. This is Case (3).

We can treat the case where $|P|=3^2$ in the same way.

Next we look at the situation where $P$ has maximal class, which is Lemma \ref{pstructure}~(3). Since neither (1) nor (2) of the lemma holds, we find a $P$-orbit of size $3$ or $\frac{|P|}{3}$.
With an orbit of size $3$ we may suppose that $|P:H|=3$ and we can argue as in the previous case.
Otherwise, there is an orbit of size $\frac{|P|}{3}$ and we let $\alpha \in \Omega$ be such that $|\alpha^P|=\frac{|P|}{3}$ and hence $|P_\alpha|=3$. 
Our hypotheses imply that the point stabilisers are t.i. subgroups of $G$ and we see that Lemma 1.9 from \cite{PS} applies. Then we can follow the arguments from \cite{MW3} again, leading to Case (1) of the lemma.  
\end{proof}

\begin{prop}\label{new6.3} (New version of \cite[Proposition 6.3]{MW3})
\mbox{}

Suppose that Hypothesis \ref{hypkfix} holds and let $\omega \in \Omega$. Suppose that $G_\omega$ has odd order and that $|\FO(G_\omega)|=3$. Then one of the following
is true:

(1) $G$ has a regular normal subgroup.

(2) $G$ has a normal subgroup $F$ of index $3$ which acts as a Frobenius group on each of its
three orbits.

(3) $G$ has a normal subgroup $N$ which acts semi-regularly, $G/N$ is
almost simple, $G_\omega$ is cyclic of order at least $4$ and if $N$ is non-trivial, then $NG_\omega$ acts as a Frobenius group on $\omega^N$.
\end{prop}

\begin{proof}
The hypothesis implies that every point stabiliser is a three point stabiliser, which is a stronger hypothesis than asking for $G$ to be a $(0,3)$-group. It implies that $G_\omega$ is a t.i. subgroup, which is why we deviate from \cite{MW3} and omit Case (1) of the original proposition. 

Since this shows up several times in the proof,
we briefly discuss the case where $G$ is soluble.
Theorem 1.3 in \cite{HMW} gives two possibilities with point stabilisers of odd order. The first one is exactly our Case (2), and the second one is almost our case (3). If $G$ has a subgroup of index 6 that acts like a Frobenius group on its three orbits, then the orbit stabiliser theorem yields that point stabilisers must have even order. Therefore, this situation only occurs with a subgroup of index 3.
This concludes the special case where $G$ is soluble. 

From now on, we suppose that $G$ is a minimal counterexample to the proposition and we let $N$ denote a minimal normal subgroup.

\smallskip
First we assume for a contradiction that $N$ is abelian.

\smallskip
Let $p$ be a prime such that $N$ is an elementary abelian $p$-group. 
Then we can work under Hypothesis 4.1 in \cite{HMW} with all its notation, and Lemma 4.2 in \cite{HMW} gives that
$N$ acts semi-regularly on $\Omega$, because the point stabilisers have odd order.
Lemma 4.3 tells us that $\bar G$ acts with fixity at most 3 on $\bar \Omega$, and then we can look at the possibilities depending on the size of $\bar \Omega$.
If $|\bar \Omega| \le 2$, then Lemma 4.7 of \cite{HMW} applies, and (ii) yields that $G$ has a regular normal subgroup, as in (2).
If $|\bar \Omega|=3$, then Lemma 4.8 of \cite{HMW} gives that one of the cases (2) or (3) of our proposition holds, bearing in mind that a Frobenius subgroup of index 6 is not possible.
Now suppose that  $|\bar \Omega|\ge 4$. We recall that $\bar G$ acts transitively and with fixity at most 3 on $\bar \Omega$, in particular the action is faithful, and we can now look at each possible fixity.
Since $G_\omega$ stabilises $\omega^N=\bar \omega$, we see that the smallest case is fixity 1. In this case $\bar G$ acts like a Frobenius group and the full pre-image of the Frobenius kernel is a regular normal subgroup of $G$, as in (2).

Fixity 2 comes with strong restrictions: Assume for a contradiction that $\bar G$ acts with fixity 2 on $\bar \Omega$ and let $\alpha, \beta, \omega$ denote the three fixed points of $G_\omega$. Since $G_\omega$ stabilises every $N$-orbit that contains one of the elements $\alpha, \beta$ or $\omega$, these three elements must be distributed onto the exactly two $N$-orbits that $G_\omega$ stabilises. It is not possible that one of the orbits contains one fixed point and the other one contains two, which leads to the following situation (with loss): $\alpha, \beta \in \omega^N$ and $G_\omega$ also stabilises another $N$-orbit, but without fixing any points. This forces $|G_\omega|=3$ and $p=3$, and if we look at Lemma \ref{solumin}~(b), then the only possibility is $|N|=3$. 

Now in $G$, the fact that $G_\omega$ has order 3 and fixes three points forces $|N_G(G_\omega)|=9$, whereas in $\bar G$, the fact that $\bar G_\omega$ fixes exactly two points on $\bar \Omega$ forces $|N_{\bar G}(\bar G_\omega)|=6$, with full pre-image $L$ in $G$ of order 18.
Since $L$ must stabilise the union of $G_\omega$-invariant $N$-orbits, we see that it acts faithfully and with fixity 3 on a set of size 6.
Then Theorem 1.3 in \cite{HMW} gives a contradiction.\\

If the fixity is 3, then we 
check if we can apply our proposition to $\bar G$. 
We see that $\overline{G_\omega}$ has odd order and we 
let $x\in G_\omega N$ be such that $\bar x$ fixes three points on $\bar \Omega$. In particular $x \notin N$ and without loss $x \in G_\omega$. There are two cases, in parallel to the fixity 2 situation in the previous case:  

Case 1: The three fixed points of $x$ lie in $\omega^N$ and $x$ stabilises two more $N$-orbits, but fixed-point-freely. This implies that $|G_\omega|=3$ and $|N|=3$ as before, and each element of $G_\omega$ stabilises the same three $N$-orbits.

Case 2: The three fixed points of $x$ lie in three distinct $N$-orbits. This implies that each element of $G_\omega$ stabilises the same three $N$-orbits.

We conclude that $\bar G$ satisfies the hypothesis of our proposition in the action on $\bar \Omega$, and since $G$ is a minimal counterexample, we may apply (1)--(3).
In Case (1) we take the full pre-image in $G$ of a regular normal subgroup of $\bar G$, which gives (1).

In Case (2) we have a subgroup $\bar F$ of $\bar G$ with three orbits on  $\bar \Omega$ that acts as a Frobenius group on each of them. As we have seen before, the fact that the point stabilisers have odd order forces $\bar G$, and then $G$, to be soluble. 

In Case (3) we find a semi-regular normal subgroup $\bar M$ of $\bar G$ such that $\bar G/\bar M$ is almost simple and point stabilisers are cyclic of order at least $4$. The full pre-image $M$ of $\bar M$ in $G$ is a semi-regular normal subgroup then, $G/M$ is almost simple and
the point stabilisers in $G$ are cyclic of order at $4$ as well because they intersect $M$ trivially. If there exists a non-trivial $x \in G_\omega$ with more than one fixed point in $\omega^M$, then we get the contradiction $|G_\omega| = 3$. Hence $M G_\omega$ acts as a Frobenius group on $\omega^M$.

Since $G$ is a counterexample, none of this happens and $N$ cannot be abelian. In particular $F(G)=1$.

\smallskip
Then $N$ is a product of simple components of $G$, and Theorem \ref{min} tells us that $N$ is a unique component and therefore $N=F^*(G)$.
In particular $G$ is almost simple. If $G$ is simple, then Theorem 1.1 in \cite{MW3} together with the hypothesis that $|G_\omega|$ is odd implies that $G_\omega$ is cyclic of order coprime to $6$. In particular, $|G_\omega| > 3$. If $G$ is almost simple but not simple, then Theorem 1.2 in \cite{MW3} is applicable. Cases (ii) and (iii) imply that the point stabilisers are cyclic of order at least $4$. Note that $\PGU_3(2)$ is not almost simple. In Case (i), $G \cong \aut(\PSL_2(2^p))$ for some prime $p$ with point stabilisers of order $2^p \cdot (2^p - 1) \cdot p$. 
In particular, the point stabilisers have even order in contradiction to our hypothesis.

\end{proof}



\vspace{0.5cm}
\textbf{Proof of Theorem \ref{main3}:}

For the assertion that $3 \in \pi(G)$, the $2$-structure of $G$ plays a role, and all arguments are correct. We have already revisited the places where Lemma 2.23 and Theorem 2.24 of \cite{MW3} are used in the previous section. 
Then, following the strategy in \cite{MW3}, we consider the case where point stabilisers have even order, replacing Proposition 6.2 of \cite{MW3} by our Proposition \ref{6.2} above. The possibilities there give exactly the cases (1)~(a) -- (d) of Theorem \ref{main3}.

If the point stabilisers have odd order, then we may suppose that $|\Omega| \ge 7$, checking smaller cases with \texttt{GAP}, and then Lemma 2.5 of \cite{MW3} is applicable. Let $\omega \in \Omega$.
If $G_\omega$ is a Frobenius group and the Frobenius complements are three point stabilisers, then Corollary 2.6 from \cite{MW3} applies. Let $H$ denote a Frobenius complement in $G_\omega$. Then $(G, G/H)$ satisfies Hypothesis \ref{hypkfix} with the natural action of $G$ by right multiplication, and more specifically the point stabilisers are three point stabilisers. This already contributes to the hypothesis of Proposition \ref{new6.3} for this action -- we only need to see that the point stabilisers have odd order.
Hence let $g \in G$. Then $H^g$ is the stabiliser of the point $Hg$, and $H^g$ has odd order because $H \le G_\omega$ has odd order.
Consequently, Proposition \ref{new6.3} is applicable to the action of $G$ on $G/H$.
First we let $N$ denote a normal subgroup of $G$ that acts regularly on $G/H$, as in Case (1).
Then we turn to Lemma 6.4 from \cite{MW3}, which is about exactly this situation, and it tells us that $G$ is soluble and gives specific information about the structure of $G$. This is described in Lemma 2.10 of \cite{MW3}. In particular we have a regular normal subgroup in the action on $\Omega$,
as stated in Theorem \ref{main3}~(2)~(a).

If Proposition \ref{new6.3}~(2) holds, then we let $F$ denote a normal subgroup of index $3$ of $G$ that has three orbits on $G/H$ and acts on them like a Frobenius group.
First we note that $H \le F$ because otherwise $G$ is not transitive on $G/H$, and then $H$ is contained in a point stabiliser in $F$.=
Conversely, if $L \le F$ is a Frobenius complement, then it fixes an element of $G/H$ and is therefore contained in a conjugate of $H$. Therefore, the conjugates of $H$ are exactly the Frobenius complements in $F$. 
We also see that $G_\omega \le F$: Otherwise $G_\omega \cap F$ is a normal subgroup of index $3$ in $G_\omega$, and it contains $H$, which is a Frobenius complement in $F$. This is impossible. 
Let $K \le F$ denote the Frobenius kernel of $F$ in the action on $G/H$. Then $K_\omega \unlhd G_\omega$ is contained in the Frobenius kernel of $G_\omega$. Conversely, if $x \in G_\omega$ is contained in the Frobenius kernel in $G_\omega$, then it fixes only $\omega$ and no other point in $\Omega$. In particular it is not contained in any conjugate of $H$, but it is contained in $F$. Then it follows that $x\in K$, hence $x\in K_\omega$.

We finish this case by investigating the action of $K$ on $\Omega$. Since $K$ is characteristic in $F$ and hence normal in $G$, it cannot be contained in a point stabiliser. Therefore $K_\omega \lneq K$. 
If $|K:K_\omega| \le 3$, then $|\omega^K| \le 3$ and therefore $|\Omega|\le 9$. Now $K$ is isomorphic to a nilpotent subgroup of $\Sym_9$, $1 \neq K_\omega$ has odd order and $G_\omega$ has odd order, which forces $K$ to be a $3$-group. But then the Frobenius complements in $F$ must have even order, which is impossible.
If $|K:K_\omega| > 3$, then we recall that $K$ acts with fixity $2$ or $3$ on $\Omega$ and, since $K$ is nilpotent, it follows that $K$ has prime power order. Moreover, $|G_\omega|$ is odd and therefore $K$ is a $3$-group of maximal class. This implies that $|K:\Phi(K)| = 9$ \cite[III.14.2]{Hupp} and again it follows that the Frobenius complements have even order. This gives a contradiction.

Finally, if Proposition \ref{new6.3}~(3) holds, then $G$ has a semi-regular normal subgroup $N$ (in the action on $G/H$) such that $\overline{G} := G/N$ is almost simple and the point stabilisers are cyclic, which in our situation means that $H$ is cyclic. We note that, therefore, $G_\omega$ is a Frobenius group of odd order with cyclic complements.

First, we assume that $N_\omega = N \cap G_\omega \neq 1$. Then $1 \neq N_\omega \neq N$ because the action is faithful. Moreover, since $NH$ is a Frobenius group with kernel $N$, we also know that $N$ is nilpotent. As in the previous case, we see that this is only possible if $|H|$ is even.

If follows that $N_\omega = 1$ and therefore $\overline{G_\omega} \cong G_\omega$ is a Frobenius group. If $N \neq 1$, then $N H$ is a Frobenius group with complement $H$ by Proposition \ref{new6.3} and $G_\omega$ is also a Frobenius group with complement $H$ by our choice of $H$. If follows that $N G_\omega$ is a Frobenius group with complement $H$ and Lemma \ref{normaliser} (i) yields $|N| \leq 3$. This forces $H$ to have even order in contradiction to our hypothesis.

Hence $N = 1$ and $G = \overline{G}$ is almost simple. As in the proof of Proposition \ref{new6.3}, we go through the different possibilities in Theorem 1.1 and 1.2 of \cite{MW3} and see that $G_\omega$ must be cyclic.


\end{document}